\documentclass{article}

\usepackage{amsfonts}
\usepackage{amsopn}
\usepackage{amsmath}
\usepackage{amsthm}
\usepackage{latexsym}
\usepackage{amssymb}
\usepackage{graphicx}

\newcommand{\MOMSdgl}[1]{\left\{\begin{aligned}
		#1
	\end{aligned}\right.}
\newcommand{\MOMSnsX}{{n_{\scalebox{0.42}{$X$}}}}
\newcommand{\MOMSnsY}{{n_{\scalebox{0.42}{$Y$}}}}
\newcommand{\MOMSnsXi}{{n_{\scalebox{0.42}{$\Xi$}}}}
\newcommand{\MOMSdbar}{\, \mathrm{d}\llap{\raisebox{0.85ex}{$\scriptstyle-\!$}}}
\newcommand{\MOMSd}{\, \mathrm{d}}
\newcommand{\MOMSsklammer}[1]{\left\langle#1\right\rangle}
\newcommand{\MOMSnorm}[1]{\left\|#1\right\|}
\newcommand{\MOMStranspop}[1]{{}^{\mathsf t} \! {#1} }
\newcommand{\MOMSmenge}[1]{\left\lbrace #1\right\rbrace }
\newcommand{\MOMSabs}[1]{\left|#1\right|}

\DeclareMathOperator{\dist}{dist}
\DeclareMathOperator{\sign}{sign}
\DeclareMathOperator{\supp}{supp}

\newtheorem{theorem}{Theorem}[section]
\newtheorem{lemma}[theorem]{Lemma}
\newtheorem{remark}[theorem]{Remark}
\newtheorem{corollary}[theorem]{Corollary}
\newtheorem{proposition}[theorem]{Proposition}
\newtheorem{definition}[theorem]{Definition}
\newtheorem{example}[theorem]{Example}

\begin{document}

\title{On the Measurability of Stochastic Fourier Integral Operators}

\author{Michael Oberguggenberger\thanks{Unit of Engineering Mathematics, University of Innsbruck,
Technikerstra\ss e 13, 6020 Innsbruck,
Austria, (michael.oberguggenberger@uibk.ac.at)}
 \and
Martin Schwarz\thanks{Unit of Engineering Mathematics, University of Innsbruck,
Technikerstra\ss e 13, 6020 Innsbruck,
Austria, (martin.schwarz@uibk.ac.at)}
}

\date{}
\maketitle

\abstract{This work deals with the measurability of Fourier integral operators (FIOs) with random phase and amplitude functions. The key ingredient is the proof that FIOs depend continuously on their phase and amplitude functions, taken from suitable classes. The results will be applied to the solution FIO of the transport equation with spatially random transport speed as well as to FIOs describing waves in random media.}


\section{Introduction}
In the theory of hyperbolic partial differential equations (PDEs), Fourier integral operators (FIOs) have become an important tool to examine certain properties of the solution, e.g., the propagation of singularities \cite{BoggiattoP1996,MascarelloM1997,METaylor1981}. When studying waves in random media, the coefficients of the underlying PDEs are random fields. This has become important in seismology \cite{Fouque2007, NairWhite1991} and in material science \cite{LOS2017, MOberguggenberger2018}. As the phase and amplitude functions of the FIOs producing a solution or a parametrix are functions of the coefficients of the underlying PDE, one has to ensure that the FIO, respectively its action, stays measurable. The question of measurability of a FIO arises also in its own right, when a deterministic phase or amplitude function is subjected to a stochastic perturbation \cite{MOberguggenberger2014}. This work is dedicated to providing a rigorous proof of various continuity and measurability properties.

Consider a FIO of the form
\begin{align*}
A_{\Phi,a} [u]({\boldsymbol{x}}) =\frac{1}{(2\pi)^n}\int_{}\int_{} {\mathrm{e}}^{i\Phi({\boldsymbol{x}},{\boldsymbol{y}},{\boldsymbol{\xi}})} a({\boldsymbol{x}},{\boldsymbol{y}},{\boldsymbol{\xi}}) u({\boldsymbol{y}}) \MOMSd {\boldsymbol{y}} \MOMSd {\boldsymbol{\xi}}.
\end{align*}
The first task will be to show that, for $\psi\in {\mathcal D}(Y)$, respectively $u\in {\mathcal E}'(Y)$, the maps
\begin{equation}\label{eq:cont}
   (\Phi,a) \mapsto A_{\Phi,a} [\psi],\quad (\Phi,a) \mapsto A_{\Phi,a} [u]
\end{equation}
are continuous with values in ${\mathcal C}^\infty(X)$, respectively ${\mathcal D}'(X)$, on suitable spaces of phase and amplitude functions. (Here $X$ and $Y$ are open subsets of ${\mathbb{R}}^{\MOMSnsX}$ and ${\mathbb{R}}^{\MOMSnsY}$.) Equipping these spaces with their Borel $\sigma$-algebra, we consider random functions
\[ \omega\to (\Phi_\omega({\boldsymbol{x}},{\boldsymbol{y}},{\boldsymbol{\xi}}),a_\omega({\boldsymbol{x}},{\boldsymbol{y}},{\boldsymbol{\xi}})) \]
on a probability space $(\Omega,{\mathcal{F}},{\mathbb{P}})$. The second task will be to infer that the maps
\[ \omega\mapsto A_{\Phi_\omega,a_\omega} [\psi], \quad \omega\mapsto A_{\Phi_\omega,a_\omega} [u]\]
are measurable as well. The applicability of these results will be demonstrated in three examples: the transport equation and the half wave equation, both with a spatially random propagation speed, and a random perturbation of the solution operator to the wave equation.

The plan of the paper is as follows. We start by recalling required facts from the classical theory of Fourier integral operators. In the subsequent section, we prove the continuity and measurability results. The final section addresses the announced applications. The paper is part of a larger program aiming at studying wave propagation in random media by means of stochastic Fourier integral operators \cite{MOberguggenberger2014, MPSchwarz2019}; it provides the probabilistic basis for this program.

\section{Classical theory of oscillatory integrals and FIOs}
\subsection{Oscillatory integrals} \label{sec:osc_int}

Let $\MOMSnsY,\MOMSnsXi\in{\mathbb{N}}$ and let $Y\subset{\mathbb{R}}^{\MOMSnsY}$ be an open set. The subsequent short exposition follows \cite{MAShubin1987}. Let
\begin{align}\label{eqn:osc_int}
I_\Phi(au)=\int_{{\mathbb{R}}^{\MOMSnsXi}}\int_{Y} {\mathrm{e}}^{i\Phi({\boldsymbol{y}},{\boldsymbol{\xi}})} a({\boldsymbol{y}},{\boldsymbol{\xi}}) u({\boldsymbol{y}}) \MOMSd {\boldsymbol{y}} \MOMSd {\boldsymbol{\xi}} .
\end{align}
Here $u\in{\mathcal{D}}(Y)$ is a smooth function with compact support and $\Phi:Y\times{\mathbb{R}}^{\MOMSnsXi}$ is a \emph{phase function}, which means that $\Phi\big|_{Y\times({\mathbb{R}}^{\MOMSnsXi}\backslash \boldsymbol{\MOMSmenge{0}})}$ is smooth, real valued and positively homogeneous of degree $1$ in ${\boldsymbol{\xi}}$. Furthermore, $\Phi$ does not have any critical points in ${\mathbb{R}}^{\MOMSnsXi}\backslash {\MOMSmenge{\boldsymbol 0}}$, i.e., for all ${\boldsymbol{y}}\in Y$ and ${\boldsymbol{\xi}}\in{\mathbb{R}}^{\MOMSnsXi}\backslash {\MOMSmenge{\boldsymbol 0}}$
\[
[\partial_{y_1} \Phi({\boldsymbol{y}},{\boldsymbol{\xi}}),\ldots,\partial_{y_{\MOMSnsY}} \Phi({\boldsymbol{y}},{\boldsymbol{\xi}}),\partial_{\xi_1} \Phi({\boldsymbol{y}},{\boldsymbol{\xi}}),\ldots,\partial_{\xi_{\MOMSnsXi}} \Phi({\boldsymbol{y}},{\boldsymbol{\xi}})]^{{\mathsf{T}}} \neq {\boldsymbol 0}.
\]
The function $a:Y\times{\mathbb{R}}^{\MOMSnsXi}$ is a H{\"o}rmander symbol of class $S_{\varrho,\delta}^d(Y\times{\mathbb{R}}^{\MOMSnsXi})$, $d\in{\mathbb{R}}$, $0\leq\delta<1$ and $0< \varrho\leq 1$. That means, it is smooth and for any given multi-indices ${\boldsymbol{k}}\in{\mathbb{N}}^{\MOMSnsY},{\boldsymbol{l}}\in{\mathbb{N}}^{\MOMSnsXi}$  and any compact $K \subset Y$ there exists a constant $C_{{\boldsymbol{l}},{\boldsymbol{k}},K}$ such that
\[ \MOMSabs{\partial_{{\boldsymbol{\xi}}}^{{\boldsymbol{l}}} \partial_{{\boldsymbol{y}}}^{{\boldsymbol{k}}} \ a({\boldsymbol{y}},{\boldsymbol{\xi}})}\leq C_{{\boldsymbol{l}},{\boldsymbol{k}},K} \MOMSsklammer{{\boldsymbol{\xi}}}^{d-\varrho\MOMSabs{{\boldsymbol{l}}}+\delta\MOMSabs{{\boldsymbol{k}}}},\]
where ${\boldsymbol{y}}\in K$ and ${\boldsymbol{\xi}} \in{\mathbb{R}}^{\MOMSnsXi}$. As usual, we write $\MOMSsklammer{{\boldsymbol{\xi}}}=(1+\MOMSnorm{{\boldsymbol{\xi}}}^2)^{1/2}$.

In general, $({\mathrm{e}}^{i\Phi({\boldsymbol{y}},{\boldsymbol{\xi}})} a({\boldsymbol{y}},{\boldsymbol{\xi}}) u({\boldsymbol{y}}))$ is not absolutely integrable; the oscillatory integral \eqref{eqn:osc_int} has to be regularized. We recall the usual procedure, as presented, e.g., in \cite{MAShubin1987}.

Let $\chi\in{\mathcal{D}}({\mathbb{R}}^\MOMSnsXi)$ with $\chi({\boldsymbol{\xi}})\equiv 1$ for $\MOMSnorm{{\boldsymbol{\xi}}}<1$ and $\chi({\boldsymbol{\xi}})\equiv0$ for $\MOMSnorm{{\boldsymbol{\xi}}}>2$. Furthermore, let
\begin{equation}\label{eqn:the_equation_of_r}
 r({\boldsymbol{\xi}},{\boldsymbol{y}})=\bigg(\sum_{l=1}^{\MOMSnsXi} \MOMSnorm{{\boldsymbol{\xi}}}^2 \MOMSabs{\partial_{ \xi_l}\Phi({\boldsymbol{y}},{\boldsymbol{\xi}}) }^2+\sum_{k=1}^{n_y} \MOMSabs{\partial_{ y_k} \Phi({\boldsymbol{y}},{\boldsymbol{\xi}})}^2\bigg),
\end{equation}
and
\begin{equation}\label{eq:coeffL}
{\underline{\alpha}}_l =
\frac{-i}{r} (1-\chi) \MOMSnorm{{\boldsymbol{\xi}}}^2 (\partial_{ \xi_l} \Phi),  \quad
{\underline{\beta}}_k = \frac{-i}{r} (1-\chi) (\partial_{ y_k} \Phi),\quad
{\underline{\gamma}} =\chi.
\end{equation}
Let $L$ be the differential operator
\begin{align}\label{eqn:regularizing_operator_L_ohne_produktregel}
L f=-\sum_{l=1}^{\MOMSnsXi}{\partial_{\xi_l}} ({\underline{\alpha}}_l f  ) - \sum_{k=1}^{\MOMSnsY} {\partial_{y_k}} ({\underline{\beta}}_k f) +{\underline{\gamma}} f.
\end{align}
 Then, the formal adjoint operator $\MOMStranspop{L}$
\[ \MOMStranspop{L}=\bigg(\sum_{l=1}^{\MOMSnsXi} {\underline{\alpha}}_l{\partial_{ \xi_l}} + \sum_{k=1}^{\MOMSnsY} {\underline{\beta}}_k  {\partial_{ y_k}} +{\underline{\gamma}} \bigg), \]
satisfies
\[ \MOMStranspop{L} {\mathrm{e}}^{i\Phi}={\mathrm{e}}^{i\Phi}. \]
Furthermore, ${\underline{\alpha}}_l\in S^0_{1,0}(Y\times {\mathbb{R}}^{\MOMSnsXi}),{\underline{\beta}}_k\in S^{-1}_{1,0}(Y\times {\mathbb{R}}^{\MOMSnsXi})$ and ${\underline{\gamma}}\in S^{-1}_{1,0}(Y\times {\mathbb{R}}^{\MOMSnsXi})$.

Finally, one can choose a $\kappa\in{\mathbb{N}}$ large enough and iteratively apply $L$ to \eqref{eqn:osc_int}, and get a convergent integral by
\[ I_\Phi(au)= \int_{{\mathbb{R}}^{\MOMSnsXi}} \int_{Y} {\mathrm{e}}^{i\Phi({\boldsymbol{y}},{\boldsymbol{\xi}})} L^\kappa(a({\boldsymbol{y}},{\boldsymbol{\xi}}) u({\boldsymbol{y}})) \MOMSd {\boldsymbol{y}} \MOMSd {\boldsymbol{\xi}}. \]

\subsection{Classical theory of FIOs} \label{subsec:class_theory}
Let $\MOMSnsX,\MOMSnsY,\MOMSnsXi\in{\mathbb{N}}$ and $X$ resp. $Y$ be an open subset of ${\mathbb{R}}^{\MOMSnsX}$ resp. ${\mathbb{R}}^ {\MOMSnsY}$. A Fourier integral operator is an operator of the form
\begin{align} \label{eqn:FIO_classic}
A_{\Phi,a} [u]({\boldsymbol{x}}) =\int_{{\mathbb{R}}^{\MOMSnsXi}}\int_{Y} {\mathrm{e}}^{i\Phi({\boldsymbol{x}},{\boldsymbol{y}},{\boldsymbol{\xi}})} a({\boldsymbol{x}},{\boldsymbol{y}},{\boldsymbol{\xi}}) u({\boldsymbol{y}}) \MOMSd {\boldsymbol{y}} \MOMSdbar {\boldsymbol{\xi}},
\end{align}
where $\MOMSdbar {\boldsymbol{\xi}} = (2\pi)^{-n}\MOMSd {\boldsymbol{\xi}}$. Let $\Phi:X\times Y\times {\mathbb{R}}^{\MOMSnsXi}\to {\mathbb{R}}$ be a \emph{phase function} on $X\times Y\times {\mathbb{R}}^{\MOMSnsXi}$ and $a \in S ^{d}_{\varrho,\delta}(X\times Y \times {\mathbb{R}}^{\MOMSnsXi})$ an \emph{amplitude function}. Furthermore, the following conditions for ${\boldsymbol{\xi}}\neq0$ are assumed:
\begin{align}\label{eqn:operator_phase_function2}
[\partial_{x_1} \Phi({\boldsymbol{x}},{\boldsymbol{y}},{\boldsymbol{\xi}}),\ldots,\partial_{x_{\MOMSnsX}}  \Phi({\boldsymbol{x}},{\boldsymbol{y}},{\boldsymbol{\xi}}),\partial_{\xi_1} \Phi({\boldsymbol{x}},{\boldsymbol{y}},{\boldsymbol{\xi}}),\ldots,\partial_{\xi_{\MOMSnsXi}}  \Phi({\boldsymbol{x}},{\boldsymbol{y}},{\boldsymbol{\xi}})]^{{\mathsf{T}}} &\neq \boldsymbol{0} , \\ \label{eqn:operator_phase_function1}
[\partial_{y_1}  \Phi({\boldsymbol{x}},{\boldsymbol{y}},{\boldsymbol{\xi}}),\ldots,\partial_{y_{\MOMSnsY}}  \Phi({\boldsymbol{x}},{\boldsymbol{y}},{\boldsymbol{\xi}}),\partial_{\xi_1}  \Phi({\boldsymbol{x}},{\boldsymbol{y}},{\boldsymbol{\xi}}),\ldots,\partial_{\xi_{\MOMSnsXi}}  \Phi({\boldsymbol{x}},{\boldsymbol{y}},{\boldsymbol{\xi}})]^{{\mathsf{T}}} &\neq \boldsymbol{0},
\end{align}
The function $\Phi$ is then called an \emph{operator phase function}.

By classical theory, if condition \eqref{eqn:operator_phase_function1} is satisfied, operator \eqref{eqn:FIO_classic} continuously maps ${\mathcal{D}}(Y)$ into ${\mathcal{C}}^\infty(X)$. Under condition \eqref{eqn:operator_phase_function2} operator \eqref{eqn:FIO_classic} can be extended to a continuous map from ${\mathcal{E}}'(Y)$ into ${\mathcal{D}}'(X)$ by
\[ \MOMSsklammer{A[u],\phi}=\MOMSsklammer{u,\MOMStranspop{A}[\phi]}, \]
where
\[ \MOMStranspop{A}_{\Phi,a} [\phi]({\boldsymbol{y}}) = \int_{{\mathbb{R}}^{\MOMSnsXi}} \int_{X} {\mathrm{e}}^{i\Phi({\boldsymbol{x}},{\boldsymbol{y}},{\boldsymbol{\xi}})} \ a({\boldsymbol{x}},{\boldsymbol{y}},{\boldsymbol{\xi}}) \phi({\boldsymbol{x}}) \MOMSd {\boldsymbol{x}} \MOMSdbar {\boldsymbol{\xi}}.\]

\section{Stochastic Fourier integral operators} \label{subesc:sfios}

The set of operator phase functions and the space of amplitudes are equipped with a natural metrizable topology, which we now describe. The first task in this section will be to prove the continuity of the maps \eqref{eq:cont}. In order to derive the required inequalities, it will be necessary to replace the conditions on nondegeneracy \eqref{eqn:operator_phase_function2}, \eqref{eqn:operator_phase_function1} by strict bounds away from zero. It turns out that this is not a serious restriction in the applications. The resulting  space is a closed subset of the space of operator phase functions.

To avoid technical difficulties we hereafter exclude $\boldsymbol{0}$ from ${\mathbb{R}}^{\MOMSnsXi}$ and set $\Xi={\mathbb{R}}^{\MOMSnsXi}\backslash\MOMSmenge{\boldsymbol{0}}$.
The sets $X$ and $Y$ are open subsets of ${\mathbb{R}}^{\MOMSnsX}$ and ${\mathbb{R}}^{\MOMSnsY}$, respectively.

If $Z$ is an open subset of ${\mathbb{R}}^d$, $d\in\{\MOMSnsX,\MOMSnsY\}$, we define a compact exhaustion $K_{Z,m}$ of $Z$ by
\[ K_{Z,m}=\{{\boldsymbol{x}}\in Z, \MOMSnorm{{\boldsymbol{x}}}\leq m, \ \dist({\boldsymbol{x}},{\mathbb{R}}^{d}\backslash Z)\geq  1/m\}.\]
\begin{definition}
    The space of \emph{positively homogeneous functions of degree one} is defined by
    \begin{align*}
	{\mathcal{M}}_{{hg}}(X,Y,\Xi)=\Big\{&\Phi:X\times Y \times \Xi \to {\mathbb{R}}, \Phi \text{ smooth and }\\ &  \text{ positively homogeneous of degree $1$ w.r.t. ${\boldsymbol{\xi}}$} \Big\}.
	\end{align*}
    Let $\alpha > 0$. The subspace ${\mathcal{M}}_\alpha(X,Y,\Xi)$ is given by
    \begin{align*}
	{\mathcal{M}}_\alpha(X,Y,\Xi)=\Bigg\{&\Phi\in {\mathcal{M}}_{{hg}}(X,Y,\Xi) \text{ such that } \forall {\boldsymbol{x}} \in X, {\boldsymbol{y}} \in Y, {\boldsymbol{\xi}} \in \Xi: \\
	&\MOMSnorm{\begin{pmatrix}
		\MOMSnorm{{\boldsymbol{\xi}}}^{-1} \nabla_{{\boldsymbol{x}}}\\\nabla_{{\boldsymbol{\xi}}}
		\end{pmatrix}\Phi({\boldsymbol{x}},{\boldsymbol{y}},{\boldsymbol{\xi}})}^2\geq\alpha ,  \MOMSnorm{\begin{pmatrix}
		\MOMSnorm{{\boldsymbol{\xi}}}^{-1} \nabla_{{\boldsymbol{y}}} \\\nabla_{{\boldsymbol{\xi}}}
		\end{pmatrix} \Phi({\boldsymbol{x}},{\boldsymbol{y}},{\boldsymbol{\xi}})}^2\geq \alpha   \  \Bigg\} .
	\end{align*}
\end{definition}
	Let ${\boldsymbol{j}}\in{\mathbb{N}}^{\MOMSnsX},{\boldsymbol{k}}\in{\mathbb{N}}^{\MOMSnsY},{\boldsymbol{l}}\in{\mathbb{N}}^{\MOMSnsXi}$ be multi-indices and let  $(p_m)_{m\in\mathbb{N}}$ be the sequence of seminorms with
	\begin{align} \label{eqn:seminorms_phase}
	\begin{aligned}
	&p_m (\Phi) :=\sup\bigg\{ \MOMSnorm{\boldsymbol{\xi}}^{\boldsymbol{l}-1} \MOMSabs{\partial_{{\boldsymbol{x}}}^{{\boldsymbol{j}}}\partial_{{\boldsymbol{y}}}^{{\boldsymbol{k}}}\partial_{{\boldsymbol{\xi}}}^{{\boldsymbol{l}}} \Phi\left({\boldsymbol{x}},{\boldsymbol{y}},{{\boldsymbol{\xi}}} \right) }, {\boldsymbol{x}}\in K_{X,m}, {\boldsymbol{y}}\in K_{Y,m} ,\\
	&\qquad\qquad\qquad\qquad\qquad\qquad\qquad\qquad\qquad {\boldsymbol{\xi}}\in\Xi, \MOMSabs{{\boldsymbol{j}}}+\MOMSabs {\boldsymbol{k}}+\MOMSabs {\boldsymbol{l}}\leq m\bigg\}.
	\end{aligned}
	\end{align}
The topology of ${\mathcal{M}}_{{hg}}(X,Y,\Xi)$ is defined by the seminorms $(p_m)_{m\in\mathbb{N}}$.

\begin{remark}
	Since $\Phi$ is positively homogeneous of degree one with respect to $\boldsymbol{\xi}$, an equivalent sequence of seminorms is given by
	\begin{align}
	\begin{aligned}
	&\widetilde{p}_m (\Phi) =\sup\bigg\{ \big|{\partial_{{\boldsymbol{x}}}^{{\boldsymbol{j}}}\partial_{{\boldsymbol{y}}}^{{\boldsymbol{k}}}\partial_{{\boldsymbol{\eta}}}^{{\boldsymbol{l}}} \Phi\left({\boldsymbol{x}},{\boldsymbol{y}},{{\boldsymbol{\eta}}} \right)\big| \Big|_{\boldsymbol{\eta}=\boldsymbol{\xi}/{\MOMSnorm{\boldsymbol{\xi}}}}}, {\boldsymbol{x}}\in K_{X,m}, {\boldsymbol{y}}\in K_{Y,m} ,\\
	&\qquad\qquad\qquad\qquad\qquad\qquad\qquad\qquad\qquad {\boldsymbol{\xi}}\in\Xi, \MOMSabs{{\boldsymbol{j}}}+\MOMSabs {\boldsymbol{k}}+\MOMSabs {\boldsymbol{l}}\leq m\bigg\}.
	\end{aligned}	
	\end{align}	%
\end{remark}

\begin{remark}
	Note that any function in ${\mathcal{M}}_\alpha$ automatically satisfies \eqref{eqn:operator_phase_function2} and \eqref{eqn:operator_phase_function1}. However, there are functions satisfying \eqref{eqn:operator_phase_function2} and \eqref{eqn:operator_phase_function1}, but there is no $\alpha>0$ such that they are in ${\mathcal{M}}_\alpha$. For $X=Y={\mathbb{R}}$ consider
	\[ \Phi(x,y,\xi)=({\mathrm{e}}^{-x^2}-y)\ \xi. \]
	Then $\partial_\xi \Phi(x,y,\xi)= {\mathrm{e}}^{-x^2}$ and $(\MOMSabs{\xi}^{-1 }\partial_x \Phi(x,y,\xi))= -2{\sign{(\xi)}} x {\mathrm{e}}^{-x^2} ,$ which can be arbitrarily small for $\MOMSabs{x}$ large.
\end{remark}

Now let $\partial B=\MOMSmenge{{\boldsymbol{\xi}}\in\Xi,\MOMSnorm{{\boldsymbol{\xi}}}=1}$ be the unit sphere in ${\mathbb{R}}^{\MOMSnsXi}$. The space ${\mathcal{C}}^\infty(X\times Y\times \partial B)$ with its usual topology $\mathcal{T}$ is a separable Fr\'echet space and thus metrizable and complete \cite[Chapter XVII, Section 2]{Dieudonne1997}. We have the following result
\begin{proposition}
	The space $\big({\mathcal{M}}_{{hg}}(X,Y,\Xi),(p_m)_{m\in{\mathbb{N}}}\big)$ is isomorphic to the space $\big({\mathcal{C}}^\infty(X\times Y\times \partial B), \mathcal{T} \big)$ and hence
separable, metrizable, and complete.
\end{proposition}
	\begin{proof}
	{The isomorphism is explicitly given by
	\begin{align*}
	I: {\mathcal{M}}_{{hg}}(X,Y,\Xi) &\to {\mathcal{C}}^\infty(X\times Y\times \partial B)\\
		\Big (({\boldsymbol{x}},{\boldsymbol{y}},{\boldsymbol{\xi}}) \mapsto f({\boldsymbol{x}},{\boldsymbol{y}},{\boldsymbol{\xi}}) \Big)&\mapsto \bigg(({\boldsymbol{x}},{\boldsymbol{y}},{\boldsymbol{\xi}}) \mapsto f\bigg({\boldsymbol{x}},{\boldsymbol{y}},\frac{{\boldsymbol{\xi}}}{\MOMSnorm{{\boldsymbol{\xi}}}}\bigg) \bigg).
	\end{align*}}

\noindent
The bicontinuity of $I$ can be most easily seen by employing local spherical coordinates on $\partial B$.
\end{proof}
\begin{proposition}
	Let $\alpha>0$, then ${\mathcal{M}}_{\alpha}(X,Y,\Xi)$ is a closed subset of ${\mathcal{M}}_{{hg}}(X,Y,\Xi)$.
\end{proposition}
	\begin{proof}
		We show that \[ \MOMSnorm{\left(\begin{matrix}
			\MOMSnorm{{\boldsymbol{\xi}}}^{-1} \nabla_{{\boldsymbol{x}}}\\ \nabla_{{\boldsymbol{\xi}}}
			\end{matrix}\right) (\cdot)}^2  \] continuously maps ${\mathcal{M}}_\alpha(X,Y,\Xi)$ to ${\mathbb{R}}$: Fix $({\boldsymbol{x}},{\boldsymbol{y}},{\boldsymbol{\xi}})\in(X\times Y\times \Xi)$, and choose $m\geq1$ large enough, such that ${\boldsymbol{x}}\in K_{X,m}$ and ${\boldsymbol{y}} \in K_{Y,m}$. Then, by \eqref{eqn:seminorms_phase}
		\begin{align*}
		\MOMSabs{\MOMSnorm{{\boldsymbol{\xi}}}^{-1}\partial_{x_j} \Phi({\boldsymbol{x}},{\boldsymbol{y}},{\boldsymbol{\xi}})}\leq p_m(\Phi)
		\end{align*}
		and
		\begin{align*}
		\MOMSabs{\partial_{\xi_l} \Phi ({\boldsymbol{x}},{\boldsymbol{y}},{\boldsymbol{\xi}})}\leq p_{m} (\Phi).
		\end{align*}
		So in total\[ \MOMSnorm{\left(\begin{matrix}
			\MOMSnorm{{\boldsymbol{\xi}}}^{-1} \nabla_{{\boldsymbol{x}}}\\ \nabla_{{\boldsymbol{\xi}}}
			\end{matrix} \right)\Phi({\boldsymbol{x}},{\boldsymbol{y}},{\boldsymbol{\xi}})}^2\leq (\MOMSnsXi +\MOMSnsX) (p_{m}(\Phi))^2.  \]
		Since the $2$-norm is continuous as well, the limit of a convergent sequence $(\Phi_n)_{n\in{\mathbb{N}}}$ can be pulled out
		\[  \MOMSnorm{\begin{pmatrix}
			\MOMSnorm{{\boldsymbol{\xi}}}^{-1} \nabla_{{\boldsymbol{x}}}\\\nabla_{{\boldsymbol{\xi}}}
			\end{pmatrix}  \lim_{n\to\infty}\Phi_n({\boldsymbol{x}},{\boldsymbol{y}},{\boldsymbol{\xi}})}^2 =\lim_{n\to\infty}\MOMSnorm{\begin{pmatrix}
			\MOMSnorm{{\boldsymbol{\xi}}}^{-1} \nabla_{{\boldsymbol{x}}}\\\nabla_{{\boldsymbol{\xi}}}
			\end{pmatrix}  \Phi_n({\boldsymbol{x}},{\boldsymbol{y}},{\boldsymbol{\xi}})}^2  \geq\alpha, \]
		and the continuity is shown. Since the preimage of a closed set is closed,
		\[ \MOMSmenge{\Phi\in {\mathcal{M}}_{{hg}}(X,Y,\Xi):  \MOMSnorm{\left(\begin{matrix}
				\MOMSnorm{{\boldsymbol{\xi}}}^{-1} \nabla_{{\boldsymbol{x}}}\\ \nabla_{{\boldsymbol{\xi}}}
				\end{matrix}\right) \Phi({\boldsymbol{x}},{\boldsymbol{y}},{\boldsymbol{\xi}})}^2\geq \alpha }  \]
		is closed.
		The proof works analogously for $ \MOMSnorm{\left(\begin{matrix}
			\MOMSnorm{{\boldsymbol{\xi}}}^{-1} \nabla_{{\boldsymbol{y}}}\\ \nabla_{{\boldsymbol{\xi}}}
			\end{matrix}\right) (\cdot)}^2 $, and the intersection of closed sets is closed again.
\end{proof}

\begin{definition}
	The space of amplitude functions on $X\times Y \times \Xi$ is defined by
	\[ {\mathcal{S}}_{\varrho,\delta}^{d}(X,Y,\Xi)=\MOMSmenge{a\big|_{X\times Y\times\Xi} ,\ a\in S_{\varrho,\delta}^d(X\times Y\times {\mathbb{R}}^{\MOMSnsXi})}. \]
	This space is equipped with the seminorms
	\begin{align*}
	q_m (a) =\sup\Big\{ \MOMSsklammer{{\boldsymbol{\xi}}}^{-d+\varrho\MOMSabs{{\boldsymbol{l}}}  -\delta  (\MOMSabs{{\boldsymbol{j}}}+\MOMSabs{{\boldsymbol{k}}})} &\MOMSabs{\partial_{{\boldsymbol{x}}}^{{\boldsymbol{j}}}\partial_{{\boldsymbol{y}}}^{{\boldsymbol{k}}}\partial_{{\boldsymbol{\xi}}}^{{\boldsymbol{l}}}\ a({\boldsymbol{x}},{\boldsymbol{y}},{\boldsymbol{\xi}})}, \text{ where } \\&{\boldsymbol{x}}\in K_{X,m}, {\boldsymbol{y}}\in K_{Y,m} , {\boldsymbol{\xi}} \in \Xi,\MOMSabs {{\boldsymbol{j}}}+\MOMSabs {{\boldsymbol{k}}} + \MOMSabs{{\boldsymbol{l}}}\leq m\big\}.
	\end{align*}
\end{definition}

Note that by \cite[Chapter VII, Section 7.8]{Hoermander2003} the space $({\mathcal{S}}_{\varrho,\delta}^d,(q_m)_{m\in{\mathbb{N}}})$ forms a Fr{\'e}chet space, and thus it is complete and metrizable. Furthermore, one can check that it is separable. Actually, ${\mathcal{S}}_{\varrho,\delta}^{d}(X,Y,\Xi)$ is isomorphic with
$S_{\varrho,\delta}^d(X\times Y\times {\mathbb{R}}^{\MOMSnsXi})$.

In any case, both the phase function space and the amplitude function space are closed subsets of separable, metrizable, complete spaces. For any further consideration we will deal with
\[ {\mathcal{F}}_{\alpha,\varrho,\delta}^d(X,Y,\Xi)={\mathcal{M}}_\alpha(X,Y,\Xi)\times {\mathcal{S}}_{\varrho,\delta}^d(X,Y,\Xi),\]
which is equipped with the product topology of the two spaces, induced by the seminorms $\boldsymbol{p_m}(\Phi,a):=p_m(\Phi)+q_m(a), \ m\in\mathbb N$.
We call ${\mathcal{F}}_{\alpha,\varrho,\delta}^d(X,Y,\Xi)$ the space of \emph{FIO operator functions}.

\begin{lemma}		
	Let $\Phi\in {\mathcal{M}}_\alpha(X,Y,\Xi)$ and let ${\underline{\alpha}}_i$ resp. ${\underline{\beta}}_i$ be the coefficients of the regularizing operator $L$ (see \eqref{eq:coeffL}). Then for $m\in{\mathbb{N}}$ there exists a polynomial ${\mathcal{P}}_m$ such that for any ${\boldsymbol{x}}\in K_{X,m},{\boldsymbol{y}}\in K_{Y,m}$ and ${\boldsymbol{\xi}}\in \Xi$
	\begin{align} \label{eqn:estimate_ak_fios}
	\MOMSabs{\partial_{{\boldsymbol{x}}}^{{\boldsymbol{j}}}\partial_{{\boldsymbol{y}}}^{{\boldsymbol{k}}}\partial_{{\boldsymbol{\xi}}}^{{\boldsymbol{l}}}\ {\underline{\alpha}}_i ({\boldsymbol{x}},{\boldsymbol{y}},{\boldsymbol{\xi}})}\leq{\mathcal{P}}_m(p_m(\Phi))  \MOMSsklammer{{\boldsymbol{\xi}}}^{-\MOMSabs{{\boldsymbol{l}}}}
	\end{align}
	for $i\in\MOMSmenge{1,\ldots,\MOMSnsXi}$ and
	\begin{align} \label{eqn:estimate_bj_fios}
	\MOMSabs{\partial_{{\boldsymbol{x}}}^{{\boldsymbol{j}}}\partial_{{\boldsymbol{y}}}^{{\boldsymbol{k}}} \partial_{{\boldsymbol{\xi}}}^{{\boldsymbol{l}}}\ {\underline{\beta}}_i ({\boldsymbol{x}},{\boldsymbol{y}},{\boldsymbol{\xi}})}\leq {\mathcal{P}}_m(p_m(\Phi)) \MOMSsklammer{{\boldsymbol{\xi}}}^{-1-\MOMSabs{{\boldsymbol{l}}}}
	\end{align}
	for $i\in\MOMSmenge{1,\ldots, \MOMSnsY}$ and all $\MOMSabs{{\boldsymbol{j}}}+\MOMSabs{{\boldsymbol{k}}}+\MOMSabs{{\boldsymbol{l}}}+1 \leq m$.
	\begin{proof}
		For this proof we will show the inequalities only for the zeroth and first derivative of \eqref{eqn:estimate_ak_fios}. Any higher derivative can be treated the same way.
		
		Recall the notation from Equation \eqref{eqn:the_equation_of_r}: $$r({\boldsymbol{x}},{\boldsymbol{y}},{\boldsymbol{\xi}})={\MOMSnorm{\MOMSnorm{{\boldsymbol{\xi}}} \nabla_{{\boldsymbol{\xi}}} \Phi({\boldsymbol{x}},{\boldsymbol{y}},{\boldsymbol{\xi}})}^2 +\MOMSnorm{\nabla_{{\boldsymbol{y}}} \Phi({\boldsymbol{x}},{\boldsymbol{y}},{\boldsymbol{\xi}})}^2}.$$
		Since $\Phi\in{\mathcal{M}}_\alpha(X,Y,\Xi)$ one has that
		\[ \alpha \MOMSnorm{{\boldsymbol{\xi}}}^2 \leq r({\boldsymbol{x}},{\boldsymbol{y}},{\boldsymbol{\xi}}), \]
		and therefore,
		\begin{align*}
		\MOMSabs{{\underline{\alpha}}_i({\boldsymbol{x}},{\boldsymbol{y}},{\boldsymbol{\xi}})}&=\MOMSabs{ \frac{(1-\chi({\boldsymbol{\xi}})) \MOMSnorm{{\boldsymbol{\xi}}}^2 \partial_{\xi_i}\Phi({\boldsymbol{x}},{\boldsymbol{y}},{\boldsymbol{\xi}})}{r({\boldsymbol{x}},{\boldsymbol{y}},{\boldsymbol{\xi}})}} \\&\leq \MOMSabs{  \frac{1-\chi({\boldsymbol{\xi}})}{\alpha} \ \partial_{\xi_i} \Phi({\boldsymbol{x}},{\boldsymbol{y}},{\boldsymbol{\xi}})}
		\\&\leq  {C_m^0 p_m(\Phi)} {=:{\mathcal{P}}_m^0(p_m(\Phi))}
		\end{align*}
		
		We note that for all multi-indices ${\boldsymbol{j}},{\boldsymbol{k}}$ and ${\boldsymbol{l}}$ with $\MOMSabs{{\boldsymbol{j}}}+\MOMSabs{{\boldsymbol{k}}}+\MOMSabs{{\boldsymbol{l}}}+1 \leq m$ there exists a constant $K_m$ such that
		\[
		\MOMSabs{\partial_{{\boldsymbol{x}}}^{{\boldsymbol{j}}}\partial_{{\boldsymbol{y}}}^{{\boldsymbol{k}}}\partial_{{\boldsymbol{\xi}}}^{{\boldsymbol{l}}}\ r({\boldsymbol{x}},{\boldsymbol{y}},{\boldsymbol{\xi}})}
		\leq K_m \MOMSnorm{{\boldsymbol{\xi}}}^{2-\MOMSabs{{\boldsymbol{l}}}} (p_m(\Phi))^2\leq K_m \MOMSsklammer{{\boldsymbol{\xi}}}^{2-\MOMSabs{{\boldsymbol{l}}}} (p_m(\Phi))^2,
		\] for ${\boldsymbol{\xi}}\in\Xi,\MOMSnorm{{\boldsymbol{\xi}}}\geq1.$ Therefore,
		\begin{align*}
		&\MOMSabs{\partial_{y_k} {\underline{\alpha}}_i({\boldsymbol{x}},{\boldsymbol{y}},{\boldsymbol{\xi}})}\\
		=\ &\MOMSabs{\partial_{y_k} \frac{(1-\chi({\boldsymbol{\xi}})) \MOMSnorm{{\boldsymbol{\xi}}}^2 \partial_{\xi_i}\Phi({\boldsymbol{x}},{\boldsymbol{y}},{\boldsymbol{\xi}})}{r({\boldsymbol{x}},{\boldsymbol{y}},{\boldsymbol{\xi}})}}\\
		\leq\ &  \MOMSabs{(1-\chi({\boldsymbol{\xi}}))\partial_{y_k} \frac{\MOMSnorm{{\boldsymbol{\xi}}}^2 \partial_{\xi_i}\Phi({\boldsymbol{x}},{\boldsymbol{y}},{\boldsymbol{\xi}})}{r({\boldsymbol{x}},{\boldsymbol{y}},{\boldsymbol{\xi}})}}\\
		\leq\ & \MOMSabs{ (1-\chi({\boldsymbol{\xi}}))\frac{\MOMSnorm{{\boldsymbol{\xi}}}^2 \partial_{\xi_i, {y_k}}\Phi({\boldsymbol{x}},{\boldsymbol{y}},{\boldsymbol{\xi}}) r({\boldsymbol{x}},{\boldsymbol{y}},{\boldsymbol{\xi}})- \MOMSnorm{{\boldsymbol{\xi}}}^2 \partial_{\xi_i }\Phi({\boldsymbol{x}},{\boldsymbol{y}},{\boldsymbol{\xi}}) \partial_{y_k}r({\boldsymbol{x}},{\boldsymbol{y}},{\boldsymbol{\xi}})}{r^2({\boldsymbol{x}},{\boldsymbol{y}},{\boldsymbol{\xi}})}}\\
		\leq\ &  \MOMSabs{ (1-\chi({\boldsymbol{\xi}}))\frac{ \partial_{\xi_i,{y_k}}\Phi({\boldsymbol{x}},{\boldsymbol{y}},{\boldsymbol{\xi}}) r({\boldsymbol{x}},{\boldsymbol{y}},{\boldsymbol{\xi}})}{\alpha^2 \MOMSnorm{{\boldsymbol{\xi}}}^2}}+  \MOMSabs{ (1-\chi({\boldsymbol{\xi}}))\frac{\partial_{\xi_i }\Phi({\boldsymbol{x}},{\boldsymbol{y}},{\boldsymbol{\xi}}) \partial_{y_k}r({\boldsymbol{x}},{\boldsymbol{y}},{\boldsymbol{\xi}})}{\alpha^2 \MOMSnorm{{\boldsymbol{\xi}}}^2}}\\
		\leq\ & {C_m^1 (p_m(\Phi))^3}=:{{\mathcal{P}}_m^1(p_m(\Phi))}
		\end{align*}
		where the last inequality is due to the fact that $\partial_{\xi_i }\partial_{y_k}\Phi({\boldsymbol{x}},{\boldsymbol{y}},{\boldsymbol{\xi}})$ and $\partial_{\xi_i}\Phi({\boldsymbol{x}},{\boldsymbol{y}},{\boldsymbol{\xi}})$ are bounded by a constant times $p_m(\Phi)$ applying \eqref{eqn:seminorms_phase}. Since $(1-\chi({\boldsymbol{\xi}}))$ is nonzero only for $\MOMSnorm{{\boldsymbol{\xi}}}>1$, one has no difficulties in the neighborhood of $\boldsymbol 0$.
		
		Derivation with respect to $\xi_l$ yields
		
		\[ \MOMSabs{\partial_{\xi_l} {\underline{\alpha}}_i({\boldsymbol{x}},{\boldsymbol{y}},{\boldsymbol{\xi}})}\\
		=\MOMSabs{\partial_{\xi_l} \frac{(1-\chi({\boldsymbol{\xi}})) \MOMSnorm{{\boldsymbol{\xi}}}^2 \partial_{\xi_i}\Phi({\boldsymbol{x}},{\boldsymbol{y}},{\boldsymbol{\xi}})}{r({\boldsymbol{x}},{\boldsymbol{y}},{\boldsymbol{\xi}})}}, \]
		which is less or equal to
		\begin{align*}
		\leq \ &\MOMSabs{ \partial_{\xi_l} \chi({\boldsymbol{\xi}}) \frac{\MOMSnorm{{\boldsymbol{\xi}}}^2 \partial_{\xi_i}\Phi({\boldsymbol{x}},{\boldsymbol{y}},{\boldsymbol{\xi}})}{r({\boldsymbol{x}},{\boldsymbol{y}},{\boldsymbol{\xi}})}}
		+ \MOMSabs{(1-\chi({\boldsymbol{\xi}})) \frac{ 2\xi_l \partial_{\xi_i} \Phi({\boldsymbol{x}},{\boldsymbol{y}},{\boldsymbol{\xi}}) }{r({\boldsymbol{x}},{\boldsymbol{y}},{\boldsymbol{\xi}})}}\\
		&+ \MOMSabs{(1-\chi({\boldsymbol{\xi}})) \frac{ \MOMSnorm{{\boldsymbol{\xi}}}^2 \partial_{\xi_i, {\xi_l}}\Phi({\boldsymbol{x}},{\boldsymbol{y}},{\boldsymbol{\xi}}) }{r^({\boldsymbol{x}},{\boldsymbol{y}},{\boldsymbol{\xi}})}} +  \MOMSabs{(1-\chi({\boldsymbol{\xi}}))\frac{ \MOMSnorm{{\boldsymbol{\xi}}}^2 \partial_{\xi_i }\Phi({\boldsymbol{x}},{\boldsymbol{y}},{\boldsymbol{\xi}}) \partial_{\xi_l}r({\boldsymbol{x}},{\boldsymbol{y}},{\boldsymbol{\xi}})}{r^2({\boldsymbol{x}},{\boldsymbol{y}},{\boldsymbol{\xi}})}}\\	
		\leq\ &\MOMSabs{ \partial_{\xi_l} \chi({\boldsymbol{\xi}}) \frac{\MOMSnorm{{\boldsymbol{\xi}}}^2 \partial_{\xi_i}\Phi({\boldsymbol{x}},{\boldsymbol{y}},{\boldsymbol{\xi}})}{\alpha \MOMSnorm{{\boldsymbol{\xi}}}^2}}
		+ \MOMSabs{(1-\chi({\boldsymbol{\xi}})) \frac{ 2\xi_l \partial_{\xi_i} \Phi({\boldsymbol{x}},{\boldsymbol{y}},{\boldsymbol{\xi}}) }{\alpha \MOMSnorm{{\boldsymbol{\xi}}}^2}}\\
		&+ \MOMSabs{(1-\chi({\boldsymbol{\xi}})) \frac{ \MOMSnorm{{\boldsymbol{\xi}}}^2 \partial_{\xi_i, {\xi_l}}\Phi({\boldsymbol{x}},{\boldsymbol{y}},{\boldsymbol{\xi}}) }{\alpha \MOMSnorm{{\boldsymbol{\xi}}}^2}} +  \MOMSabs{(1-\chi({\boldsymbol{\xi}}))\frac{ \MOMSnorm{{\boldsymbol{\xi}}}^2 \partial_{\xi_i }\Phi({\boldsymbol{x}},{\boldsymbol{y}},{\boldsymbol{\xi}}) \partial_{\xi_l}r({\boldsymbol{x}},{\boldsymbol{y}},{\boldsymbol{\xi}})}{\alpha^2 \MOMSnorm{{\boldsymbol{\xi}}}^4}}\\
		\leq \ & \big(C^2_m \ p_m(\Phi)+C^3_m  (p_m(\Phi))^3\big) \MOMSsklammer{\boldsymbol{\xi}}^{-1}=:{\mathcal{P}}_m^2(p_m(\Phi)) \ \MOMSsklammer{{\boldsymbol{\xi}}}^{-1} ,
		\end{align*}
		where we used that $\xi_l/\MOMSnorm{{\boldsymbol{\xi}}}\leq 1$ and the same arguments as before.

		Having done the estimation for all ${\boldsymbol{j}},{\boldsymbol{k}},{\boldsymbol{l}}$, in the end one can set
		\[{\mathcal{P}}_m=\sum_{i} {\mathcal{P}}_m^i,\]
		since all coefficients of ${\mathcal{P}}_m^i$ are nonnegative. To prove \eqref{eqn:estimate_bj_fios} one can use the same arguments.
	\end{proof}
\end{lemma}

\begin{definition}
	The seminorms $(\pi_{X,{m}})_{m\in{\mathbb{N}}}$ for ${\mathcal{C}}^\infty(X)$ are defined by
	\[ \pi_{X,{m}} (v) =\sup\MOMSmenge{ \MOMSabs{\partial_{{\boldsymbol{x}}}^{{\boldsymbol{j}}}\ v({\boldsymbol{x}})}, {\boldsymbol{x}}\in K_{X,{m}}, \MOMSabs{{\boldsymbol{j}}}\leq {m}}, \qquad v\in{\mathcal{C}}^\infty(X).\]	
\end{definition}

\begin{proposition}
	Let $(\Phi_n,a_n)_{n\in{\mathbb{N}}}$ be a convergent sequence in the product space $(
	{\mathcal{F}}_{\alpha,\varrho,\delta}^d(X,Y,\Xi),(\boldsymbol{p}_m)_{m\in{\mathbb{N}}})$ with limit $(\Phi,a)$. Furthermore, let $L_n$ be the regularizing operator (cf. Section \ref{sec:osc_int}) for $(\Phi_n,a_n)$. Let $\widetilde{m}\in{\mathbb{N}}$ and $\psi\in{\mathcal{D}}(Y)$ be fixed. Furthermore, let $\MOMSabs{{\boldsymbol{j}}}\leq \widetilde{m}$.
	Choose $\kappa\in{\mathbb{N}}$ large enough, such that
	\[ \MOMSabs{\partial_{{\boldsymbol{x}}}^{{\boldsymbol{j}}} \left({\mathrm{e}}^{i\Phi_n({\boldsymbol{x}},{\boldsymbol{y}},{\boldsymbol{\xi}})} L^\kappa_n(a_n({\boldsymbol{x}},{\boldsymbol{y}},{\boldsymbol{\xi}}) \psi({\boldsymbol{y}}))\right)}={\mathcal{O}}(\MOMSsklammer{{\boldsymbol{\xi}}}^{-(\MOMSnsXi+1)}). \]
	 Choose $m$ large enough such that $\kappa+\widetilde{m}+1\leq m$ and $\supp (\psi)\subset K_{Y,m}$. Then
	
	\begin{enumerate}
		\item[(a)] there exists a constant ${C}_{m,p_m(\Phi_n),q_m(a_n),\psi}$, depending on $m,p_m(\Phi_n),q_m(a_n)$ and $\psi$, such that
		\begin{align*}
		&\sup_{{\boldsymbol{x}}\in K_{X,\widetilde{m}}, \ {\boldsymbol{y}} \in Y} \MOMSabs{\partial_{{\boldsymbol{x}}}^{{\boldsymbol{j}}} \left({\mathrm{e}}^{i\Phi_n({\boldsymbol{x}},{\boldsymbol{y}},{\boldsymbol{\xi}})} L^\kappa_n(a_n({\boldsymbol{x}},{\boldsymbol{y}},{\boldsymbol{\xi}}) \psi({\boldsymbol{y}}))\right)}\\&\qquad \qquad \qquad \qquad \leq {C}_{m,p_m(\Phi_n),q_m(a_n),\psi} \MOMSsklammer{{\boldsymbol{\xi}}}^{-(\MOMSnsXi+1) }  .
		\end{align*}
		\item[(b)] the sequence of FIOs $A_{\Phi_n,a_n}$ converges in the following sense:
		\[ \lim_{n\to\infty} \pi_{X,\widetilde{m}} (A_{\Phi_n,a_n}[\psi] - A_{\Phi,a}[\psi] ) )=0, \]
		where
		\[ A_{\Phi_n,a_n}[\psi]({\boldsymbol{x}})=\int_{\Xi} \int_{Y} {\mathrm{e}}^{i\Phi_n({\boldsymbol{x}},{\boldsymbol{y}},{\boldsymbol{\xi}})} L^\kappa_n(a_n({\boldsymbol{x}},{\boldsymbol{y}},{\boldsymbol{\xi}}) \psi({\boldsymbol{y}}))  \MOMSd {\boldsymbol{y}}\MOMSdbar {\boldsymbol{\xi}}, \]
		and
		\[ A_{\Phi,a}[\psi]({\boldsymbol{x}})=\int_{\Xi} \int_{Y} {\mathrm{e}}^{i\Phi({\boldsymbol{x}},{\boldsymbol{y}},{\boldsymbol{\xi}})} L^\kappa(a({\boldsymbol{x}},{\boldsymbol{y}},{\boldsymbol{\xi}}) \psi({\boldsymbol{y}}))  \MOMSd {\boldsymbol{y}}\MOMSdbar {\boldsymbol{\xi}}. \]
	\end{enumerate} 	
	
	\begin{proof}
		(a) We will examine only $a_n \psi$ and $L_n(a_n \psi)$. The term $L^\kappa_n(a_n \psi)$ can be estimated in the same way, but it is much more tedious.
		
		Since $\psi$ has compact support, there exists a constant $C_{\psi,m}$, depending on $m$ and $\psi$ such that
		\[ \MOMSabs{\partial_{\boldsymbol{y}}^{{\boldsymbol{k}}} \psi({\boldsymbol{y}})} \leq C_{\psi,m},\]
		for any $\MOMSabs{{\boldsymbol{k}}}\leq  {m}$.
		By assumption, $\supp (\psi) \subset K_{Y,m}$ and therefore
		\begin{align*}
		\sup_{{\boldsymbol{x}}\in K_{X,\widetilde{m}},\ {\boldsymbol{y}} \in Y}\MOMSabs{a_n({\boldsymbol{x}},{\boldsymbol{y}},{\boldsymbol{\xi}})\psi({\boldsymbol{y}})} =\ & \sup_{{\boldsymbol{x}}\in K_{X,\widetilde{m}},\ {\boldsymbol{y}} \in K_{Y,m}}\MOMSabs{a_n({\boldsymbol{x}},{\boldsymbol{y}},{\boldsymbol{\xi}})\psi({\boldsymbol{y}})}\\ \leq \ & C_{\psi,m} \ q_m(a_n) \MOMSsklammer{{\boldsymbol{\xi}}}^d  \\\leq \ & {C}^0_{m,p_m(\Phi_n),q_m(a_n),\psi} \MOMSsklammer{{\boldsymbol{\xi}}}^{d} .
		\end{align*}
		Furthermore,
		\begin{align}
		&\ \ \ L_n(a_n({\boldsymbol{x}},{\boldsymbol{y}},{\boldsymbol{\xi}}) \psi({\boldsymbol{y}})) \nonumber\\
		&\begin{aligned}
		=&-\sum_{l=1}^{\MOMSnsXi}  \partial_{\xi_l}( {{\underline{\alpha}}}_{nl}({\boldsymbol{x}},{\boldsymbol{y}},{\boldsymbol{\xi}})  a_n({\boldsymbol{x}},{\boldsymbol{y}},{\boldsymbol{\xi}}) \psi({\boldsymbol{y}})) -\sum_{k=1}^{\MOMSnsY} \partial_{ y_k} ({{\underline{\beta}}}_{nk}({\boldsymbol{x}},{\boldsymbol{y}},{\boldsymbol{\xi}})  a_n({\boldsymbol{x}},{\boldsymbol{y}},{\boldsymbol{\xi}}) \psi({\boldsymbol{y}}))\\&+{{\underline{\gamma}}}_{n}({\boldsymbol{x}},{\boldsymbol{y}},{\boldsymbol{\xi}})  a_n({\boldsymbol{x}},{\boldsymbol{y}},{\boldsymbol{\xi}}) \psi({\boldsymbol{y}}). \label{eqn:thm:conv_fio}
		\end{aligned}
		\end{align}
		Using again that $\supp \psi\subset K_{y,m}$, the first term of \eqref{eqn:thm:conv_fio} can be bounded by
		\begin{align*}
		&\sup_{{\boldsymbol{x}}\in K_{X,\widetilde{m}}, \ {\boldsymbol{y}} \in Y}\MOMSabs{\partial_{ \xi_l} ({{\underline{\alpha}}}_{nl}({\boldsymbol{x}},{\boldsymbol{y}},{\boldsymbol{\xi}})  a_n({\boldsymbol{x}},{\boldsymbol{y}},{\boldsymbol{\xi}}) \psi({\boldsymbol{y}}))}\\
		\leq&\	\sup_{{\boldsymbol{x}}\in K_{X,\widetilde{m}}, \ {\boldsymbol{y}} \in K_{Y,m}}\MOMSabs{\partial_{ \xi_l}{{\underline{\alpha}}}_{nl}({\boldsymbol{x}},{\boldsymbol{y}},{\boldsymbol{\xi}})  a_n({\boldsymbol{x}},{\boldsymbol{y}},{\boldsymbol{\xi}}) \psi({\boldsymbol{y}})}\\
		& +	\sup_{{\boldsymbol{x}}\in K_{X,\widetilde{m}}, \ {\boldsymbol{y}} \in K_{Y,m}}\MOMSabs{{{\underline{\alpha}}}_{nl}({\boldsymbol{x}},{\boldsymbol{y}},{\boldsymbol{\xi}})  \partial_{ \xi_l} a_n({\boldsymbol{x}},{\boldsymbol{y}},{\boldsymbol{\xi}}) \psi({\boldsymbol{y}})}\\
		\leq&\  {\mathcal{P}}_m(p_m({\Phi_n})) \MOMSsklammer{{\boldsymbol{\xi}}} ^{-1}\MOMSabs{ a_n({\boldsymbol{x}},{\boldsymbol{y}},{\boldsymbol{\xi}})}\ C_{\psi,m}   \\
		& + { {\mathcal{P}}_m(p_m(\Phi_n))  \MOMSabs{\partial_{\xi_l}  a_n({\boldsymbol{x}},{\boldsymbol{y}},{\boldsymbol{\xi}})}  C_{\psi,m}  }\\
		\leq &\ C_{\psi,m}\ {\mathcal{P}}_m(p_m(\Phi_n)) \ q_m(a_n) \left( \MOMSsklammer{{\boldsymbol{\xi}}}^{d-1}+\MOMSsklammer{{\boldsymbol{\xi}}}^{d-\varrho} \right)  .
		\end{align*}
		The second term of \eqref{eqn:thm:conv_fio} can be bounded by
		\begin{align*}
		&	\sup_{{\boldsymbol{x}}\in K_{X,\widetilde{m}}, \ {\boldsymbol{y}} \in Y}\MOMSabs{\partial_{ y_k} ({{\underline{\beta}}}_{nk}({\boldsymbol{x}},{\boldsymbol{y}},{\boldsymbol{\xi}})  a_n({\boldsymbol{x}},{\boldsymbol{y}},{\boldsymbol{\xi}}) \psi({\boldsymbol{y}}))}\\
		\leq&\	\sup_{{\boldsymbol{x}}\in K_{X,\widetilde{m}}, \ {\boldsymbol{y}} \in K_{Y,m}}\MOMSabs{\partial_{ y_k}{{\underline{\beta}}}_{nk}({\boldsymbol{x}},{\boldsymbol{y}},{\boldsymbol{\xi}})  a_n({\boldsymbol{x}},{\boldsymbol{y}},{\boldsymbol{\xi}}) \psi({\boldsymbol{y}})}\\
		&+	\sup_{{\boldsymbol{x}}\in K_{X,\widetilde{m}}, \ {\boldsymbol{y}} \in K_{Y,m}}\MOMSabs{{{\underline{\beta}}}_{nk}({\boldsymbol{x}},{\boldsymbol{y}},{\boldsymbol{\xi}})  \partial_{ y_k} a_n({\boldsymbol{x}},{\boldsymbol{y}},{\boldsymbol{\xi}}) \psi({\boldsymbol{y}})}\\
		& +	\sup_{{\boldsymbol{x}}\in K_{X,\widetilde{m}}, \ {\boldsymbol{y}} \in K_{Y,m}}\MOMSabs{{{\underline{\beta}}}_{nk}({\boldsymbol{x}},{\boldsymbol{y}},{\boldsymbol{\xi}})  a_n({\boldsymbol{x}},{\boldsymbol{y}},{\boldsymbol{\xi}})\partial_{ y_k} \psi({\boldsymbol{y}})}\\
		\leq&\  {\mathcal{P}}(p_m({\Phi_n})) \MOMSsklammer{{\boldsymbol{\xi}}} ^{-1} \MOMSabs{a_n({\boldsymbol{x}},{\boldsymbol{y}},{\boldsymbol{\xi}})} \ C_{\psi,m}   \\
		& + { {\mathcal{P}}_m(p_m(\Phi_n)) \MOMSsklammer{{\boldsymbol{\xi}}} ^{-1} \partial_{y_k}  \MOMSabs{a_n({\boldsymbol{x}},{\boldsymbol{y}},{\boldsymbol{\xi}})} C_{\psi,m} }\\
		& + { {\mathcal{P}}_m(p_m(\Phi_n)) \MOMSsklammer{{\boldsymbol{\xi}}} ^{-1}  \MOMSabs{a_n({\boldsymbol{x}},{\boldsymbol{y}},{\boldsymbol{\xi}})} \ C_{\psi,m} }\\
		\leq &\  C_{\psi,m} \ {\mathcal{P}}_m(p_m(\Phi_n)) q_m(a_n) \left(\MOMSsklammer{{\boldsymbol{\xi}}}^{d-1}+\MOMSsklammer{{\boldsymbol{\xi}}}^{d-1+\delta}+ \MOMSsklammer{{\boldsymbol{\xi}}}^{d-1}\right) .
		\end{align*}
		Since ${{\underline{\gamma}}}_{n}({\boldsymbol{x}},{\boldsymbol{y}},{\boldsymbol{\xi}})\equiv\chi({\boldsymbol{\xi}})$, which is compactly supported, there exists a constant ${\widetilde{C}}_m$ such that
		\[ \sup_{{\boldsymbol{x}}\in X,{\boldsymbol{y}}\in Y ,{\boldsymbol{\xi}}\in\Xi} \MOMSabs{{{\underline{\gamma}}}_{n}({\boldsymbol{x}},{\boldsymbol{y}},{\boldsymbol{\xi}})  a_n({\boldsymbol{x}},{\boldsymbol{y}},{\boldsymbol{\xi}}) \psi({\boldsymbol{y}})}\leq  {\widetilde{C}}_m  C_{\psi,m}\ q_m(a_n) \MOMSsklammer{{\boldsymbol{\xi}}}^{-(\MOMSnsXi+1)} . \]
		So for
		\[ {C}^1_{m,p_m(\Phi_n),q_m(a_n),\psi}= C_{\psi,m} \big(\MOMSnsXi  + \MOMSnsY    +{\widetilde{C}}_m\big) \ q_m(a_n) \ \big(  {\mathcal{P}}_m(p_m(\Phi_n)) \vee 1\big) \]
		it holds that
		\[\sup_{{\boldsymbol{x}}\in K_{X,\widetilde{m}}} \MOMSabs{L_n(a_n({\boldsymbol{x}},{\boldsymbol{y}},{\boldsymbol{\xi}}) \psi({\boldsymbol{y}}))} \leq {C}^1_{m,p_m(\Phi_n),q_m(a_n),\psi} \MOMSsklammer{{\boldsymbol{\xi}}}^{(d-((1-\delta)\wedge \varrho)}. \]
		Iterative application of $L$ decreases the growth with respect to ${\boldsymbol{\xi}}$. For the $\kappa$th application of $L$ we get a constant ${C}^\kappa_{m,p_m(\Phi_n),q_m(a_n),\psi}$ such that
		\[\sup_{{\boldsymbol{x}}\in K_{X,\widetilde{m}}} \MOMSabs{L_n^\kappa(a_n({\boldsymbol{x}},{\boldsymbol{y}},{\boldsymbol{\xi}}) \psi({\boldsymbol{y}}))} \leq {C}^\kappa_{m,p_m(\Phi_n),q_m(a_n),\psi} \MOMSsklammer{{\boldsymbol{\xi}}}^{d-\kappa((1-\delta)\wedge \varrho)}. \]
		If $\kappa$ is large enough, $d-\kappa((1-\delta)\wedge \varrho)\leq-(\MOMSnsXi+1)$. In the end we set
$${C}_{m,p_m(\Phi_n),q_m(a_n),\psi}=\sum_{i=1}^\kappa {C}^i_{m,p_m(\Phi_n),q_m(a_n),\psi}.$$
		
		(b) To prove this, we would like to use the dominated convergence theorem. So first we note that if $(\Phi_n,a_n)_{n\in{\mathbb{N}}}$ converges in $({\mathcal{F}}_{\alpha,\varrho,\delta}^d(X,Y,\Xi),(\boldsymbol{p}_m)_{m\in{\mathbb{N}}})$, then $$ \sup_{x\in K} \partial_{{\boldsymbol{x}}}^{{\boldsymbol{j}}} \big( {\mathrm{e}}^{i\Phi_n({\boldsymbol{x}},{\boldsymbol{y}},{\boldsymbol{\xi}})} L_n^{\kappa} (a_n({\boldsymbol{x}},{\boldsymbol{y}},{\boldsymbol{\xi}}) \psi({\boldsymbol{y}}))\big)$$ converges pointwise for all  ${\boldsymbol{j}}\in{\mathbb{N}}^{\MOMSnsX}$, ${\boldsymbol{y}}\in Y$ and ${\boldsymbol{\xi}}\in\Xi$ and $K\subset X$ compact.
		
		Furthermore, since $(\Phi_n,a_n)_{n\in{\mathbb{N}}}$ is convergent,
		\[ C_{m,\text{sup}}=\sup_{n\in{\mathbb{N}}}\MOMSmenge{ C_{m,p_m(\Phi_n),q_m(a_n),\psi} } \]
		is finite.
		
		So, since $\psi$ is supported in $K_{Y,m}$ we see that

		\begin{align*}
		&\int_{\Xi} \int_{Y} \sup_{x\in K}\MOMSabs{\partial_{{\boldsymbol{x}}}^{{\boldsymbol{j}}} \left({\mathrm{e}}^{i\Phi_n({\boldsymbol{x}},{\boldsymbol{y}},{\boldsymbol{\xi}})} L^\kappa_n(a_n({\boldsymbol{x}},{\boldsymbol{y}},{\boldsymbol{\xi}}) \psi({\boldsymbol{y}}))\right)}  \MOMSd {\boldsymbol{y}}\MOMSdbar {\boldsymbol{\xi}} \\
		= \ &\int_{\Xi} \int_{K_{Y,m}} \sup_{x\in K} \MOMSabs{\partial_{{\boldsymbol{x}}}^{{\boldsymbol{j}}} \left({\mathrm{e}}^{i\Phi_n({\boldsymbol{x}},{\boldsymbol{y}},{\boldsymbol{\xi}})} L^\kappa_n(a_n({\boldsymbol{x}},{\boldsymbol{y}},{\boldsymbol{\xi}}) \psi({\boldsymbol{y}}))\right)}  \MOMSd {\boldsymbol{y}}\MOMSdbar {\boldsymbol{\xi}} \\
		\leq \ & \int_{\Xi} \int_{K_{Y,m}} C_{m,\text{sup}} \MOMSsklammer{{\boldsymbol{\xi}}}^{-(\MOMSnsXi+1)}  \MOMSd {\boldsymbol{y}}\MOMSdbar {\boldsymbol{\xi}},
		\end{align*}
		by which all requirements of the dominated convergence theorem are satisfied. So
		{\allowdisplaybreaks
		\begin{align*}
		& \lim_{n\to\infty}\pi_{X,\widetilde{m}} (A_{\Phi_n,a_n}[\psi] - A_{\Phi,a}[\psi] ) \\
		= \ & \lim_{n\to\infty} \sup_{{\boldsymbol{x}}\in K_{X,\widetilde{m}}} \bigg | \partial_{{\boldsymbol{x}}}^{{\boldsymbol{j}}} \int_{\Xi} \int_{Y}  {\mathrm{e}}^{i\Phi_n({\boldsymbol{x}},{\boldsymbol{y}},{\boldsymbol{\xi}})} L^\kappa_n(a_n({\boldsymbol{x}},{\boldsymbol{y}},{\boldsymbol{\xi}}) \psi({\boldsymbol{y}}))\\
		&\qquad\qquad\qquad\qquad \qquad \ -{\mathrm{e}}^{i\Phi({\boldsymbol{x}},{\boldsymbol{y}},{\boldsymbol{\xi}})} L^\kappa(a({\boldsymbol{x}},{\boldsymbol{y}},{\boldsymbol{\xi}}) \psi({\boldsymbol{y}}))  \MOMSd {\boldsymbol{y}}\MOMSdbar {\boldsymbol{\xi}} \bigg| \\
		\leq \  & \lim_{n\to\infty}  \int_{\Xi} \int_{Y} \sup_{{\boldsymbol{x}}\in K_{X,\widetilde{m}}} \bigg | \partial_{{\boldsymbol{x}}}^{{\boldsymbol{j}}}\Big( {\mathrm{e}}^{i\Phi_n({\boldsymbol{x}},{\boldsymbol{y}},{\boldsymbol{\xi}})} L^\kappa_n(a_n({\boldsymbol{x}},{\boldsymbol{y}},{\boldsymbol{\xi}}) \psi({\boldsymbol{y}}))\\
		&\qquad\qquad\qquad\qquad \qquad \ -{\mathrm{e}}^{i\Phi({\boldsymbol{x}},{\boldsymbol{y}},{\boldsymbol{\xi}})} L^\kappa(a({\boldsymbol{x}},{\boldsymbol{y}},{\boldsymbol{\xi}}) \psi({\boldsymbol{y}})) \Big)\bigg| \MOMSd {\boldsymbol{y}}\MOMSdbar {\boldsymbol{\xi}}  \\
		=  \ &  \int_{\Xi} \int_{Y} \lim_{n\to\infty}  \sup_{{\boldsymbol{x}}\in K_{X,\widetilde{m}}}  \bigg | \partial_{{\boldsymbol{x}}}^{{\boldsymbol{j}}} \Big( {\mathrm{e}}^{i\Phi_n({\boldsymbol{x}},{\boldsymbol{y}},{\boldsymbol{\xi}})} L^\kappa_n(a_n({\boldsymbol{x}},{\boldsymbol{y}},{\boldsymbol{\xi}}) \psi({\boldsymbol{y}}))\\
		&\qquad\qquad\qquad\qquad \qquad \ -{\mathrm{e}}^{i\Phi({\boldsymbol{x}},{\boldsymbol{y}},{\boldsymbol{\xi}})} L^\kappa(a({\boldsymbol{x}},{\boldsymbol{y}},{\boldsymbol{\xi}}) \psi({\boldsymbol{y}}))\Big) \bigg| \MOMSd {\boldsymbol{y}}\MOMSdbar {\boldsymbol{\xi}} \\
		= \ &0.
		\end{align*}
		}
	\end{proof}
\end{proposition}

\begin{corollary}
	Let $\psi\in{\mathcal{D}}(Y)$ fixed. The map
	\begin{align*}
	A_{\Phi,a}: {\mathcal{F}}_{\alpha,\varrho,\delta}^d(X,Y,\Xi) &\to {\mathcal{C}}^{\infty}(X)\\
	(\Phi,a)&\mapsto \int_\Xi\int_Y {\mathrm{e}}^{i\Phi({\boldsymbol{x}},{\boldsymbol{y}},{\boldsymbol{\xi}})} a({\boldsymbol{x}},{\boldsymbol{y}},{\boldsymbol{\xi}}) \psi({\boldsymbol{y}}) \MOMSd {\boldsymbol{y}} \MOMSdbar {\boldsymbol{\xi}}
	\end{align*}
	is continuous with respect to the topologies generated by $(\boldsymbol{p}_m)_{m\in{\mathbb{N}}}$ and $(\pi_{X,m})_{m\in{\mathbb{N}}}$.
\end{corollary}

\begin{remark}
	Using the same principles of the proof, one can also show that the map
	\begin{align*}
	{\mathcal{F}}_{\alpha,\varrho,\delta}^d(X,Y,\Xi)\times {\mathcal{D}}(Y) &\to {\mathcal{C}}^\infty(X)\\ (\Phi,a,\psi)&\mapsto A_{\Phi,a}[\psi]
	\end{align*}
	is continuous.
\end{remark}

\begin{corollary}
	Let $u\in{\mathcal{E}}'(Y)$ be a compactly supported distribution. Then
	\begin{align*}
	A_{\Phi,a}: {\mathcal{F}}_{\alpha,\varrho,\delta}^d(X,Y,\Xi) &\to {\mathcal{D}}'(X)\\
	(\Phi,a)&\mapsto A_{\Phi,a}[u]
	\end{align*}
	is continuous in the weak-$\ast$ topology. As in Section \ref{subsec:class_theory} the operator $A$ is defined by its adjoint
	\begin{align*}
	\MOMSsklammer{A_{\Phi,a}[u]({\boldsymbol{x}}),\psi({\boldsymbol{x}})}&=\MOMSsklammer{u({\boldsymbol{y}}),\MOMStranspop{A}_{\Phi,a}[\psi]({\boldsymbol{y}})}\\
	&=\MOMSsklammer{u({\boldsymbol{y}}),\int_\Xi\int_X {\mathrm{e}}^{i\Phi({\boldsymbol{x}},{\boldsymbol{y}},{\boldsymbol{\xi}})} a({\boldsymbol{x}},{\boldsymbol{y}},{\boldsymbol{\xi}}) \psi({\boldsymbol{x}}) \MOMSd {\boldsymbol{x}} \MOMSdbar {\boldsymbol{\xi}}}.
	\end{align*}
	
\end{corollary}
\begin{theorem} [Stochastic Fourier integral operators] \label{thm:sFIO_measurable}
	Let $(\Omega,{\mathcal{F}},{\mathbb{P}})$ be a probability space and
	\[%
	\begin{aligned}
	a: \Omega&\to  S^d_{\varrho,\delta}(X\times Y \times \Xi) \\\omega&\mapsto \big( {\boldsymbol{x}},{\boldsymbol{y}},{\boldsymbol{\xi}})\mapsto a({\boldsymbol{x}},{\boldsymbol{y}},{\boldsymbol{\xi}},\omega) \big)
	\end{aligned}
	\]
	and
	\[%
	\begin{aligned}
	\Phi:\Omega&\to{\mathcal{M}}_{\alpha}(X\times Y \times \Xi)\\\omega&\mapsto \big( ({\boldsymbol{x}},{\boldsymbol{y}},{\boldsymbol{\xi}})\mapsto \Phi({\boldsymbol{x}},{\boldsymbol{y}},{\boldsymbol{\xi}},\omega) \big)
	\end{aligned}
	\] be measurable mappings, where the target spaces are equipped with the Borel $\sigma$-algebra. Then, for fixed $\psi\in{\mathcal{D}}(X)$ resp. $u\in{\mathcal{E}}'(X)$ the mappings
	\[ %
	\begin{aligned}
	A_{\Phi,a}:\Omega &\to {\mathcal{C}}^\infty(X)\\\omega &\mapsto A_{\Phi,a}[\psi]
	\end{aligned} \quad \text{resp.} \quad 	\begin{aligned}A_{\Phi,a}:
	\Omega &\to {\mathcal{D}}'(X)\\\omega &\mapsto A_{\Phi,a}[u]
	\end{aligned}
	\]
	are measurable with respect to the Borel $\sigma$-algebra, induced by the corresponding topology.
\end{theorem}

\section{Applications}
In this section we apply Theorem \ref{thm:sFIO_measurable} in three typical situations.
\begin{example}
	Let $X=Y={\mathbb{R}}$, $\Xi={\mathbb{R}}\backslash\MOMSmenge{0}$, $\alpha>0$ be a real constant and
	\[ {\mathcal{C}}^\infty_{\alpha}(X)=\MOMSmenge{f\in{\mathcal{C}}^\infty(X): \forall x\in X: f(x)\geq \alpha}, \]
	and let
	\begin{align*}
	c: \Omega&\to {\mathcal{C}}^\infty_{\alpha}(X)\\\omega &\mapsto \big( x\mapsto c(\omega,x))
	\end{align*}
	be a measurable map. The transport equation with speed $c$ is then
	\begin{align} \label{eqn:transport_eqn}
	\MOMSdgl{\big(\partial_t+c(\omega,x)\partial_x\big)\,u(\omega,x,t) &=0\\u(\omega,x,0)&=u_0(x).}
	\end{align}
	The characteristic curves $\gamma$ satisfy
	\begin{align} \label{eqn:char_eqn_phase_func_stoch}
	\MOMSdgl{ \frac{\MOMSd}{\MOMSd \tau} \gamma(\omega,x,t;\tau) &= c (\omega,\gamma(\omega,x,t;\tau) ) \\ \gamma(\omega,x,t;t)&=x.}
	\end{align}
	Then one can check by classical ODE theory that for fixed $t\in{\mathbb{R}},\omega\in\Omega$ the map
	\begin{align*}
	S:  {\mathcal{C}}^\infty_{\alpha}(X)&\to {\mathcal{C}}^\infty(X)\\
	(x\mapsto c(\omega,x) ) &\mapsto \big(x  \mapsto\gamma(\omega,x,t;0) \big)
	\end{align*}
	is continuous in the $(\pi_{X,\widetilde{m}})_{\widetilde{m}\in{\mathbb{N}}}$-sense.
	Therefore, for $t\in{\mathbb{R}}$ fixed,
	\begin{align*}
	\Phi_\omega: \Omega & \to {\mathcal{M}}_\alpha(X,Y,\Xi)\\
	\omega& \mapsto \big((x,y,\xi) \mapsto \xi\cdot (\gamma(\omega,x,t;0)-y) \big)
	\end{align*}
	is measurable again. Furthermore, let $a(x,y,\xi)\equiv1$ and let $u_0\in{\mathcal{D}}(X)$ resp. $u_0\in{\mathcal{E}}'(X)$. Then, for fixed $\omega$, the solution to \eqref{eqn:transport_eqn} is given by $A_{\Phi_\omega,a}[u_0]$, and the maps
	\[ %
	\begin{aligned}
	\Omega &\to {\mathcal{C}}^\infty(X)\\\omega &\mapsto A_{\Phi_\omega,a}[u_0]
	\end{aligned} \quad \text{resp.} \quad 	\begin{aligned}
	\Omega &\to {\mathcal{D}}'(X)\\\omega &\mapsto A_{\Phi_\omega,a}[u_0]
	\end{aligned}
	\]
	are measurable.
\end{example}

\begin{example}
	Consider as before $X=Y={\mathbb{R}}$, $\Xi={\mathbb{R}}\backslash\MOMSmenge{0}$. The half wave equation in one dimension is
	\begin{align}\label{eqn:half_wave_eqn}
	\MOMSdgl{\big(\partial_t +i c(\omega,x) P(D_x)\big)\, u(\omega,x,t)&=0\\u(\omega,x,0&)=u_0(x),}
	\end{align}
	where $P$ is the pseudodifferential operator with symbol $P(\xi)=\MOMSabs{\xi}(1-\chi(\xi)),$ where $\chi\in{\mathcal{D}}({\mathbb{R}})$, $\chi(\xi)\equiv1$ for $\MOMSabs{\xi}<1/4$ and $\chi(\xi)\equiv0$ for $\MOMSabs{\xi}>1/2$. We assume that $c:\Omega\mapsto{\mathcal{C}}^\infty_\alpha(X)$ is measurable and that $c(\omega,x)$ and $\MOMSabs{\partial_x c(\omega,x)}$ are bounded from above by a constant $R$ for any $x\in{\mathbb{R}}$ and $\omega\in\Omega$. Further, all seminorms $\pi_{X,{m}}(c(\omega,\cdot))$ are assumed to be bounded independently of $\omega$.
	
	We fix $\omega\in\Omega$ for now and for notational simplicity we will drop it. A FIO parametrix for \eqref{eqn:half_wave_eqn} can be constructed following \cite[Chapter VIII, §3]{METaylor1981}. The phase function of the parametrix is of the form $\Phi(x,y,\xi)=\MOMSabs{\xi}\phi\big(x,\frac{\xi}{\MOMSabs{\xi}},t\big)-y\cdot\xi$, where $\phi$ satisfies the eikonal equation,
	\begin{align}
	\MOMSdgl{ \partial_t\phi(x,t) + c(x) P(\partial_x \phi(x,t))&=0\\ \phi(x,0)  &=x \cdot \xi.}
	\end{align}
	This is a nonlinear differential equation of the form $Q(x,t,\phi_x,\phi_t)=0$.  In \cite[Chapter II, Sect 19]{FTreves1975} one can find an explicit representation of the solution to this problem, which we will write down in the following.
	Let $F(t;x_1,\xi_1)$ and $G(t;x_1,\xi_1)$ satisfy the following equations
	\begin{align}
	\MOMSdgl{\frac{\MOMSd F}{\MOMSd t}&= c(F) P'(G)\\ F(0;x_1,\xi_1)&=x_1}
	\qquad \MOMSdgl{\frac{\MOMSd G}{\MOMSd t}&= -c'(F) P(G)\\ G(0;x_1,\xi_1)&=\xi_1.}
	\end{align}
	There exists a $T\in{\mathbb{R}}$, such that
	\[ P(G(t;x_1,\xi_1))\equiv\MOMSabs{\xi_1} \quad \text{and} \quad  P'(G(t;x_1,\xi_1))\equiv\sign{\xi_1} \]
	for $0\leq t\leq T$ and for all $\xi_1,x_1\in\mathbb{R}$, $\MOMSabs{\xi_1}\geq1$. In that case $F$ does not depend on $G$ and $x=F(t;x_1,\xi_1)$ is just the flow from $(0,x_1)$ to $(t,x)$. The inverse flow is then $F(-t,x,\xi_1)$, which goes from $(t,x)$ to $(0,x_1)$. Finally, $\phi$ is given by
	\begin{align}\label{eqn:representation_of_phi}
	\phi(x,t,\xi)=x\xi-\int_0^t c(x) P(G(F(-s;x,\xi),s,\xi))  \MOMSd s.
	\end{align}
	for $t\in[0,T]$. We observe that $F$ depends only on the sign of $\xi_1$ and $\lambda G(t;x_1,\xi_1)= G(t;x_1,\lambda\xi_1)$ for $\lambda\geq 1$, $\MOMSabs{\xi_1}>1$ and all $x\in{\mathbb{R}}$ and $t\in[0,T]$. Thus,
	\[ \phi(x,t,\xi)-\MOMSabs{\xi}\phi\bigg(x,t,\frac{\xi}{\MOMSabs{\xi}}\bigg) \]
	is compactly supported with respect to $\xi$ for all $x\in{\mathbb{R}}$ and $t\in[0,T]$.  Noting that $\phi(x,0,\xi)=x\xi$ and by the representation \eqref{eqn:representation_of_phi} one can also see that for a sufficiently small time interval, there is an $\alpha$ such that $\Phi\in{\mathcal{M}}_\alpha(X,Y,\Xi)$.
	
	Finally, by classical ODE theory, one can show that $F$ and $G$ continuously depend on $c$ and thus also $\Phi$ (at least for a short time interval).

	To obtain the amplitude one has to solve a cascade of transport equations of the form
	\begin{align*}
		\big(\partial_t-c(x)P'(\xi)\partial_x + H(x,t,\xi)\big) a_j(x,t,\xi)=0, \qquad j\leq 0,
	\end{align*}
	where $H$ is a smooth function, depending on derivatives of $c$,$\phi$ and $a_k,k>j$ (for details see again \cite[Chapter VIII, §3]{METaylor1981}). Similarly to the previous example one can show that $a_j$ continuously depends on $c$. In \cite{MAShubin1987} one can find an explicit construction of a function $a$ such that $(a-\sum_{j\leq 0} a_j(x,\xi,t))$ is smoothing, namely
	\[ a(x,t,\xi)=\sum_{j\leq0} (1-\chi(\xi/n_{j})) a_{j}(x,t,\xi), \]
	where $\chi\in{\mathcal{D}}({\mathbb{R}})$, as above, $\chi(\xi)\equiv1$ for $\MOMSabs{\xi}<1/4$ and $\chi(\xi)\equiv0$ for $\MOMSabs{\xi}>1/2$, and $n_j\in{\mathbb{N}}$ approaches $+\infty$ quickly enough as $j$ goes to $-\infty$. Since the $a_j$ depend continuously on $c$ in the $\mathcal{C}^\infty$-topology, one can chose $n_j$ such that the series converges uniformly together with all derivatives when $c$ is taken from a bounded set in $\mathcal{C}_\alpha^\infty(X)$. Thus, $a$ depends continuously on $c$ as well in the corresponding topologies.
	
	Consequently, the phase function $\Phi$ and the amplitude function $a$ of the parametrix $A_{\Phi,a}$ depend continuously on $c$ in the corresponding topologies.
	
	Recall that $c:\Omega\to\mathcal{C}_\alpha^\infty(X)$ is a random function. At fixed $\omega\in\Omega$, one obtains the parametrix $A_{\Phi_\omega, a_\omega}$. By Theorem \ref{thm:sFIO_measurable}, if $u_0\in {\mathcal{D}}({\mathbb{R}})$ resp. $u_0\in{\mathcal{E}}'({\mathbb{R}})$, then the map
	\[ \omega\to A_{\Phi_\omega,a_\omega}[u_0](x,t) \]
	is measurable.

\end{example}

\begin{example}
The standard bottom-up approach to modeling waves in random media would place the randomness in the coefficients of the underlying PDEs. However, the solution depends in a strongly nonlinear way on the coefficients of the equation, even if the PDEs are linear. This makes it hard to track or compute the stochastic features of the solution, such
as the expectation, the variance or the autocovariance function. Accordingly, a top-down approach has been proposed in \cite{MOberguggenberger2018,MPSchwarz2019}. Starting from the mean field equations as constant coefficient PDEs, the deterministic solution can readily be represented by FIOs. The stochastic properties of the medium are then modeled
\emph{a posteriori} by random perturbations of the phase and amplitude functions. In applications to material sciences, e.g., damage detection, these random perturbations can be calibrated to measurement data (just as in the bottom-up approach, where the coefficients are calibrated to measurement data). This program has been carried through in
\cite{MPSchwarz2019} in the case of three-dimensional linear elasticity.

The approach leads to stochastic FIOs, whose measurability has to be proven. We present a simplified example using the scalar wave equation.

Let $Y={\mathbb{R}}^{n}$, $X={\mathbb{R}}^n\times {\mathbb{R}}$ and $\Xi={\mathbb{R}}^n\backslash\MOMSmenge{\boldsymbol{0}}$. One can solve the deterministic wave equation with constant wave speed $c_0$
	\[ \MOMSdgl{(\partial_{tt} -c_0 \Delta) u({\boldsymbol{x}},t)&=0\\u({\boldsymbol{x}},0)&=u_0({\boldsymbol{x}}) \\ \partial_tu({\boldsymbol{x}},0)&=0} \]
	using a linear combination of FIOs with phase functions
	\[ \Phi({\boldsymbol{x}},t,{\boldsymbol{\xi}},{\boldsymbol{y}})=\langle {\boldsymbol{x}}-{\boldsymbol{y}},{\boldsymbol{\xi}} \rangle \pm c\MOMSnorm{{\boldsymbol{\xi}}} t  \]
	and amplitude function
	\[ a({\boldsymbol{x}},{\boldsymbol{y}},{\boldsymbol{\xi}})\equiv \frac{1}{2}.\]
	Then, instead of using the deterministic $\Phi$ and $a$, randomly perturbed version are introduced. One has to make sure that the randomly perturbed phase function is still an operator phase function. Thus fix $\alpha > 0$ and let
	\begin{align*}
	c: \Omega&\to {\mathcal{C}}_\alpha^\infty({\mathbb{R}}^n)\\\omega &\mapsto \big( {\boldsymbol{x}}\mapsto c(\omega,{\boldsymbol{x}}))
	\end{align*}
	be a measurable, smooth random field with ${\mathbb{E}}(c(x))\equiv c_0$. Let $\Phi_\omega({\boldsymbol{x}},t,{\boldsymbol{y}},{\boldsymbol{\xi}})=\MOMSsklammer{{\boldsymbol{x}}-{\boldsymbol{y}},{\boldsymbol{\xi}}}\pm\MOMSnorm{{\boldsymbol{\xi}}}c({\omega,\boldsymbol{x}})t$. Then, $\Phi_\omega\in{\mathcal{M}}_{\alpha}(X,Y,\Xi)$. Note that this is guaranteed, since we included the time into the image space $X$ and $\MOMSnorm{\boldsymbol{\xi}}^{-1}\MOMSabs{ \partial_t \Phi_\omega(\boldsymbol{x},t,\boldsymbol{y},\boldsymbol{\xi})}\geq \alpha$.
	
	Furthermore, let \[
	\begin{aligned}
	a: \Omega&\to {\mathcal{S}}_{\varrho,\delta}^d(X,Y,\Xi)\\
	\omega&\mapsto a_\omega({\boldsymbol{x}},t,{\boldsymbol{y}},{\boldsymbol{\xi}})
	\end{aligned}
	\] be random with ${\mathbb{E}}(a_\omega({\boldsymbol{x}},{\boldsymbol{y}},{\boldsymbol{\xi}}))=\tfrac{1}{2}$. Then, $(\Phi_\omega,a_\omega)\in{\mathcal{F}}_{\alpha,\varrho,\delta}^d(X,Y,\Xi)$ and, as above, $\omega\to A_{\Phi_\omega,a_\omega}[u_0]$ is measurable for a given $u_0\in{\mathcal{D}}(X)$ resp. $u_0\in{\mathcal{E}}'(X)$ in the corresponding topology. If $\Phi$ and $a$ are independent, the expectation, the variance and the autocovariance function of $A_{\Phi_\omega,a_\omega}[u_0]({\boldsymbol{x}})$, resp. $\langle A_{\Phi_\omega,a_\omega}[u_0],\psi\rangle$ can be computed without much difficulty, observing that ${\mathbb{E}}\big(\exp(i\Phi_\omega({\boldsymbol{x}},t,{\boldsymbol{\xi}},{\boldsymbol{y}}))\big)$ is just the characteristic function of the random variable $\Phi_\omega({\boldsymbol{x}},t,{\boldsymbol{\xi}},{\boldsymbol{y}})$, see \cite{MOberguggenberger2014}.

\end{example}

\vspace{2mm}
\noindent
{\bf Acknowledgement.}
	This work was supported by the grant P-27570-N26 of FWF (The Austrian Science Fund).

\end{document}